\documentclass[12pt,leqno,final]{amsart}
\usepackage{amsmath,amsfonts,amsthm,amssymb,verbatim,times,enumerate,ifthen,
graphics, graphicx,xcolor}
\usepackage{caption,float}
\usepackage[mathscr]{eucal}
\usepackage[utf8]{inputenc}
\usepackage[T1]{fontenc}
\usepackage{cite}

\theoremstyle{plain}
\newtheorem{Thm}{Theorem}

\theoremstyle{remark}

\def\({\left(}
\def\){\right)}
\def\[{\left[}
\def\]{\right]}
\def\<{\left<}
\def\>{\right>}

\newcommand{\eq}[1]{\eqref{E#1}}
\newcommand{\Eq}[2]{\ifthenelse{\equal{#1}{*}}
  {\begin{equation*}\begin{aligned}#2\end{aligned}\end{equation*}}
  {\begin{equation}\begin{aligned}\label{E#1}#2\end{aligned}\end{equation}}}

\begin{document}

\title{Determinants of some  pentadiagonal matrices}

\author{L\'aszl\'o Losonczi}
\address{Faculty of Economics, University of Debrecen, Hungary}
\email{losonczi08@gmail.com, laszlo.losonczi@econ.unideb.hu}

\date{\today}

\begin{abstract}
In this paper we consider pentadiagonal $(n+1)\times(n+1)$ matrices with two subdiagonals  and two superdiagonals  at distances $k$ and $2k$ from the main diagonal where $1\le k<2k\le n$. We give an explicit formula for their determinants and also consider the Toeplitz and ``imperfect'' Toeplitz versions of such matrices. Imperfectness means that the first and last $k$ elements of the main diagonal differ from  the elements in the middle. Using the rearrangement due to Egerváry and Szász we also show how these determinants can be factorized.
\end{abstract}
\maketitle

\medskip

\section{Introduction}
Matrices have a special space in mathematics. Their theory is still actively researched and used by almost every mathematician and by several scientists working in various areas.
The research on multidiagonal, in particular tridiagonal and pentadiagonal matrices, intensified in the past years. These matrices have important applications in optimization problems \cite{CC}, autoregression modelling \cite{TL}, approximation theory \cite{PW}, Gauss-Markov random processes \cite{BU}, orthogonal polynomials,  solving elliptic and parabolic PDE's with finite difference methods \cite{FSS},  inequalities (quadratic, Wirtinger, Opial's type) \cite{E}, \cite{L2}.

Let $n,k$ be given positive integers  with $1\le k<2k\le n$ and denote by $\mathcal{M}_{n}$ the set of  $n\times n$ complex matrices.
Consider the pentadiagonal matrix $A_0=(a_{ij})\in \mathcal{M}_{n+1}$
where for $i,j=0,1,\dots,n$,
$$
a_{ij}=\left\{\begin{array}{ll}
  L_j& \text{if}\; j-i=-2k, \\
  l_j& \text{if}\; j-i=-k,  \\
  d_i& \text{if}\; j-i=0, \\
  r_i& \text{if}\; j-i=k,\\
  R_i& \text{if}\; j-i=2k,\\
  0 & \text{otherwise}.
\end{array}\right.
$$
$A_0$ has two subdiagonals  and two superdiagonals at distances $k$ and $2k$ from the main diagonal. Notice that the numbering of entries starts with zero.
We  also use  $A_{n+1,k,2k}(\bf L,l,d,r,R)$ to denote the matrix $A_0$ and $D_0=D_{n+1,k,2k}(\bf L,l,d,r,R)$ to denote its determinant where the diagonal vectors $\bf L,l,d,r,R$ are   defined by
$$
\renewcommand\arraystretch{0.75}
\begin{array}{rll}
\bf{L}&=(L_0,\dots,L_{n-2k})\quad \,\bf{l}&=(l_0,\dots,l_{n-k})\\
\bf{d}&=(d_0,\dots,d_{n})&\\
\bf{R}&=(R_0,\dots,R_{n-2k})\quad \,\bf{r}&=(r_0,\dots,r_{n-k})\\
\end{array}
$$
If  $\bf L, R$ are zero vectors then our matrix  becomes a tridiagonal one denoted by $A_{n+1,k}(\bf l,d, r)$ and its determinant by $D_{n+1,k}(\bf l,d, r)$.

We shall call $A_{n+1,k,2k}(\bf L,l,d, r,R)$ a (general) $k,2k$-pentadiagonal matrix while $A_{n+1,k}(\bf l,d, r)$ will be termed as $k$-tridiagonal.

In case of Toeplitz pentadiagonal matrices the diagonal vectors are constant vectors, i.e.
\Eq{*}{
 {\bf L}&=(L,\dots,L),{\bf R}=(R,\dots,R)\in \mathbb{R}^{n+1-2k},\\
{\bf l}&=(l,\dots,l),\quad\,{\bf r}=(r,\dots,r)\in \mathbb{R}^{n+1-k},\\
{\bf d}&=(d,\dots,d)\in \mathbb{R}^{n+1},
}
and for the matrix and its determinant the notations $A_{n+1,k,2k}(L,l,d, r,R)$ and $D_{n+1,k,2k}(L,l,d, r,R)$ will be used.

Marr and Vineyard \cite{MV})  have shown that the product of two $1$-tridiagonal Toeplitz matrix is an imperfect Toeplitz matrix $A_{n+1,1,2}^{(\alpha,\beta)}$ which is related to the corresponding Toeplitz matrix by a two-step recursion. Imperfectness means that the main diagonal is changed from $(d,\dots,d)\in \mathbb{R}^{n+1}$ to
\Eq{*}{
(\underbrace{d-\alpha,\dots,d-\alpha}_{k},\underbrace{d,\dots,d}_{n+1-2k},
\underbrace{d-\beta,\dots,d-\beta}_{k})\in\mathbb{R}^{n+1}
}
where $\alpha, \beta$ are given reals.
They found the determinants of Toeplitz and imperfect Toeplitz $1,2$-pentadiagonal matrices in terms of Chebyshev polynomials of the second kind. A similar approach was used in \cite{WLZ} to find a formula for the inverse of a $1,2$-pentadiagonal Toeplitz matrix.
Imperfect $k,2k$-pentadiagonal Toeplitz matrices will be denoted by
$A_{n+1,k,2k}^{(\alpha,\beta)}(L,l,d, r,R)$, their determinants by $D_{n+1,k,2k}^{(\alpha,\beta)}(L,l,d, r,R)$.
\medskip

A number of papers studied general $1,2$-pentadiagonal matrices and their Toeplitz versions. Algorithms and recursive formulas were found for  their determinants \cite{S1969},  \cite{Ha-Elo2008_AMC}, \cite{EA}, \cite{JYL2016} and inverses \cite{Ha-Elo2008_2}. Explicit formulas were found for the determinants of symmetric \cite{Elo2018}  skew-symmetric \cite{Elo2011} and general \cite{JYL2016} Toeplitz $1,2$-pentadiagonal matrices.
Perhaps the first appearance of $k,\ell$-pentadiagonal (Toeplitz Hermitian) matrices (where the distances of the sub- and superdiagonals from the main one are $k$ and $\ell$)
 was in  Egerv\'ary and  Sz\'asz \cite{E} with $k+\ell=n+1$, while  $k$-tridiagonal matrices appeared first in \cite{El-MS2010}. In  \cite{LiLi2016} new       sub/superdiagonals were added. A graph theoretical approach can be found in \cite{dF2015,dFKL20xx} and  some extensions in \cite{McM2009,TSU2019}. An exhaustive list of recent references is given in the survey \cite{dFK2020}.

 In spite of the large numbers of papers on  pentadiagonal  matrices there are only few of them in which determinants are given in terms of the entries.

In \cite{dFL} we developed a method to reduce  the determinant of $k,\ell$-pentadiagonal matrices  to tridiagonal determinants provided that $k+\ell\ge n+1$. If $k\ge (n+1)/3$ then by this method in \cite{dFL2} we  determined the determinants of general, Toeplitz and imperfect Toeplitz $k,2k$-pentadiagonal matrices. We proved (among others)

\begin{Thm} Assuming $(n+1)/3\le k\le n/2$ the determinant $D_{n+1,k,2k}(\bf L,l,d,r,R)$ of the general $k,2k$-pentadiagonal matrix  is
\Eq{*}{
\prod\limits_{j=0}^{n-2k}\!(d_jd_{j\!+\!k}d_{j\!+\!2k}\!-\!d_jl_{j\!+\!k}r_{j\!+\!k}\!-
\!L_jR_{j}d_{j\!+\!k}
\!-\!l_jr_{j}d_{j\!+\!2k}\!+\!l_jR_{j}l_{j\!+\!k}\!+\!r_jL_{j}r_{j\!+\!k})\!\!\!
\prod\limits_{j=n+1-2k}^{k-1}\!\!\!
(d_{j}d_{j\!+k\!}\!-\!l_{j}r_{j}).
}
\end{Thm}

If the condition $k+\ell\ge n+1$ is not satisfied then the method given in \cite{dFL}  is not applicable in general. However, if $\ell=2k$ then with some modification it is applicable. In Section 2 we develop this modification and extend Theorem 1 to the case when the restriction $(n+1)/3\le k$ is dropped. Finally in Section 3 we show  how the determinants $D_{n+1,k,2k}(L,l,d, r,R)$ can be factorized and also discuss the case of Toeplitz determinants.

\section{Reduction of general $k,2k$-pentadiagonal determinants to tridiagonal ones}
Our starting point is the matrix $A_0=A_{n+1,k,2k}({\bf L,l,d, r,R})$
where $1\le k<2k\le n.$ Let $n+1=kq+p$ where $0\le p<k$ and suppose that $k<(n+1)/3$ (or $q\ge 3$).

In the first process we multiply $A_0$ by four suitable matrices such that in the product matrix the first $k$ rows and columns contains  only zeros except the diagonal and in the lower $(n+1-k)\times(n+1-k)$ block we get a pentadiagonal matrix whose structure is similar to that of $A_0$.
\medskip

(i) Let $B_1^{(l)}\in \mathcal{M}_{n+1}$ with entries
\Eq{B_1^{(l)}}{
b_{ij}^{(1,l)}=\left\{\begin{array}{rl}
1 &\hbox{if } i=j,\\
-l_j/d_j&\hbox{if } i=k,\dots,2k-1,j=i-k,\\
0 &\hbox{otherwise}
\end{array}\right.
}
and multiply $A_0$ from the left by $B_1^{(l)}$. The effect of this multiplication is the same as multiplying  the rows $0,\dots,k-1$ of $A_0$ by the numbers $-l_0/d_0,\dots,-l_{k-1}/d_{k-1}$ and adding them to the rows $k,\dots,2k-1$, respectively. The  elements $l_0,\dots,l_{k-1}$ of the $k$th  subdiagonal disappear and the diagonal elements $d_k,\dots,d_{2k-1}$ change to
\Eq{dj1}{
 d_j^{(1)}:=d_j-l_{j-k}r_{j-k}/d_{j-k} \, \, \quad \mbox{for $j=k,\dots,2k-1$}.
}
The $2k$th superdiagonal  remains unchanged, however all its elements  $R_0,\dots, R_{k-1}$ multiplied one by one by $l_0/d_0,\dots,-l_{k-1}/d_{k-1}$, respectively   move down by $k$ units and added to $r_k, \dots, r_{2k-1}$ thus these elements change to
\Eq{rj1}{
r_j^{(1)}=r_j -R_{j-k}l_{j-k}/d_{j-k}\, , \quad \text{if } j=k,\dots,2k-1.
}
Please note that the position of entries $-l_j/d_j\,(j=0,\dots,k-1)$ in the matrix $B_1^{(l)}$ is the same as the position of $l_j\,(j=0,\dots,k-1)$ in the matrix $A_{n+1,k,\ell}(\bf L,l,d,r,R)$, the position of the entries we want to eliminate. The matrices $B_1^{(r)},B_1^{(L)},B_1^{(R)}$ in the following multiplications also have similar structures.
\medskip

(ii) Multiply $B_1^{(l)}A_0$ from the right by $B_1^{(r)}\in \mathcal{M}_{n+1}$ with entries
\Eq{B_1^{(r)}}{
b_{ij}^{(1,r)}:=\left\{\begin{array}{rl}1 &\hbox{if } i=j,\\
-r_i/d_i&\hbox{if } i=0,\dots, k-1;j=i+k,\\
0 &\hbox{otherwise}
\end{array}\right.
}
The effect of this, is multiplication of  the columns $0,\dots,k-1$ of $B_1^{(l)}A_0$ by $-r_0/d_0,\dots,-r_{k-1}/d_{k-1}$ and addition of these products to the columns $k,\dots,2k-1$, respectively. The  elements $r_0,\dots,r_{k-1}$ in these columns disappear the diagonal elements remain unchanged.

The $2k$th  subdiagonal  remains unchanged, however its elements  $L_0,\dots, L_{k-1}$ multiplied one by one by $-r_0/d_0,\dots,-r_{k-1}/d_{k-1}$ respectively  move to the right by $k$ units and added to the elements $l_{2k},\dots,l_{3k-1}$, hence these elements change to
\Eq{lj1}{
l_{j}^{(1)}=l_j-L_{j-k}r_{j-k}/d_{j-k}\, , \quad \text{if }j=2k,\dots,3k-1.
}
\medskip

(iii) Next multiply $B_1^{(l)}A_0B_1^{(r)}$ from the left by $B_1^{(L)}\in \mathcal{M}_{n+1}$ with entries
\Eq{B_1^{(L)}}{
b_{ij}^{(1,L)}:=
\left\{\begin{array}{ll}1 &\hbox{if } i=j,\\
-L_j/d_j&\hbox{if } i=2k,\dots,3k-1;j=i-2k,\\
0 &\hbox{otherwise}
\end{array}\right.
}
The effect of this on the matrix $B_1^{(l)}A_0B_1^{(r)}$ is multiplication of its rows $0,\dots,n-2k$  by $-L_0/d_0,\dots,$ $-L_{n-2k}/d_{n-2k}$ and adding  these products to the rows $2k,\dots,3k-1$, respectively. The  elements $L_0,\dots,L_{k-1}$ of the $2k$th subdiagonal disappear,  the elements $d_{2k},\dots,d_{3k-1}$  of the main diagonal change to
\Eq{dj1_2}{
d_j^{(1)}=d_j-L_{j-2k}R_{j-2k}/d_{j-2k} \, , \quad \mbox{for }j=2k,\dots,3k-1.
}
\medskip

(iv) Finally multiply $B_1^{(L)}B_1^{(l)}A_0B_1^{(r)}$ from the right by  $B_1^{(R)}\in \mathcal{M}_{n+1}$ with entries
\Eq{B_1^{(R)}}{
b_{ij}^{(1,R)}:=\left\{\begin{array}{ll}1 &\hbox{if } i=j,\\
-R_i/d_i&\hbox{if } i=0,\dots,k-1;j=i+2k,\\
0 &\hbox{otherwise}
\end{array}\right.
}
The effect of this on the matrix $B_1^{(L)}B_1^{(l)}A_0B_1^{(r)}$ is multiplication of its columns $0,\dots,k-1$  by $-R_0/d_0,\dots,-R_{k-1}/d_{k-1}$ and adding these products  to the columns $2k,\dots,3k-1$ respectively. The  elements $R_0,\dots,R_{k-1}$ vanish and  the  main diagonal does not change.
\medskip

With this the first process ended. The matrices $B_1^{(L)},B_1^{(l)},B_1^{(r)},B_1^{(R)}$ are the same as in \cite{dFL, dFL2} however the transformation rules for the entries are different. In those papers during the first (and subsequent) processes the diagonal vectors $\bf{l},\bf{r}$ just shortened while the diagonal vectors $\bf{L},\bf{R}$ transformed. Here just the opposite happened: the diagonals $\bf{L},\bf{R}$ shortened and $\bf{l},\bf{r}$ transformed.

The results of the first process are summarized in
\begin{Thm} Suppose that $1\le k\le 2k\le n,$ $k<(n+1)/3$. Then
\Eq{bigo1}{
A_1:=B_1^{(L)}B_1^{(l)}A_0B_1^{(r)}B_1^{(R)}=E_1\bigoplus A_1^*
}
where $E_1\in \mathcal{M}_{k}$ is a diagonal matrix with diagonal elements $d_0,\dots,d_{k-1}$ and
\vspace*{-2mm}
\Eq{*}{
A_1^*=A_{n+1-k,k,2k}(\bf L^{(1)},l^{(1)},d^{(1)},r^{(1)},R^{(1)}).
}
\vspace*{-1mm}
is a $k,2k$-pentadiagonal matrix with  main and other diagonal vectors  (called  first iterated diagonals)
\Eq{*}{
\begin{array}{lll}
&{\bf L}^{(1)}=(L_{k},\dots,L_{n-2k})\in\mathbb{R}^{n+1-3k}\quad
&{\bf R}^{(1)}=(R_{k},\dots,R_{n-2k})\in\mathbb{R}^{n+1-3k}\\
&\,\,{\bf l}^{(1)}=(l_{k}^{(1)},\dots,l_{n-k}^{(1)})\in\mathbb{R}^{n+1-2k}\quad &\,\,{\bf r}^{(1)}=(r_{k}^{(1)},\dots,r_{n-k}^{(1)})\in\mathbb{R}^{n+1-2k}\\
&{\bf d}^{(1)}=(d_{k}^{(1)},\dots,d_{n}^{(1)})\in\mathbb{R}^{n+1-k}&
\end{array}
}
where
\vspace*{-1mm}
\Eq{iterlr}{
\begin{array}{rl}
&l_j^{(1)}=\left\{\!\!\begin{array}{lll}
&l_j-L_{j-k}r_{j-k}/d_{j-k}\quad&(j=k,\dots,2k-1),\\
&l_j\quad&(j=2k,\dots,n-k)
\end{array}\right.\\\\
&r_j^{(1)}=\left\{\!\!\begin{array}{lll}
&r_{j}-R_{j-k}l_{j-k}/d_{j-k}\quad&(j=k,\dots,2k-1)\\
&r_j\quad&(j=2k,\dots,n-k)
\end{array}\right.
\end{array}
}
and
\vspace*{-1mm}
\Eq{iterd}{
\hspace*{9mm} d_j^{(1)}=\left\{\!\!\begin{array}{lll}
&d_j-l_{j-k}r_{j-k}/d_{j-k} \quad&(j=k,\dots,2k-1),\\
&d_j-L_{j-2k}R_{j-2k}/d_{j-2k}&(j=2k,\dots,3k-1),\\
&d_j &(j=3k,\dots,n),\\
\end{array}\right.
\vspace*{-1mm}
}
\end{Thm}
\begin{proof} From the description of the four matrix multiplication  \eq{bigo1} follows while \eq{rj1}, \eq{lj1} show the correctness of \eq{iterlr} and \eq{dj1}, \eq{dj1_2} show that \eq{iterd} is valid.
\end{proof}
For the convenience  of later calculations all letters $l_i,d_i,r_i$ referring to the original matrix will be labeled with superscripts $^{(0)}$ but occasionally we omit this label. We omit from this labeling $L_i,R_i$ since  these numbers did not change during our process.

Define for $s=1,2,\dots,q-1$ the  matrices of $\mathcal{M}_{n+1}$ (for $s=1$ these coincide with the matrices \eq{B_1^{(l)}}, \eq{B_1^{(r)}}, \eq{B_1^{(L)}}, \eq{B_1^{(R)}}) which we use in further calculations by
\Eq{*}{
\begin{array}{ll}
B_s^{(l)}=\(b_{ij}^{(s,l)}\)&=\left\{\begin{array}{ll}1 &\hbox{if } i=j,\\
-l_j^{(s-1)}/d_j^{(s-1)}&\hbox{if } i=sk,\dots, (s+1)k-1;j=i-k,\\
0 &\hbox{otherwise,}
\end{array}\right.\\\\
B_s^{(r)}=\(b_{ij}^{(s,r)}\)&=\left\{\begin{array}{ll}1 &\hbox{if } i=j,\\
-r_i^{(s-1)}/d_i^{(s-1)}&\hbox{if } i=(s-1)k,\dots, sk-1;j=i+k,\\
0 &\hbox{otherwise,}
\end{array}\right.\\\\
B_s^{(L)}=\(b_{ij}^{(s,L)}\)&=\left\{\begin{array}{ll}1 &\hbox{if } i=j,\\
-L_j^{(s-1)}/d_j^{(s-1)}&\hbox{if } i=(s+1)k,\dots,(s+2)k-1;j=i-2k,\\
0 &\hbox{otherwise,}
\end{array}\right.\\\\
B_s^{(R)}=\(b_{ij}^{(s,R)}\)&=\left\{\begin{array}{ll}1 &\hbox{if } i=j,\\
-R_i^{(s-1)}/d_i^{(s-1)}&\hbox{if } i=(s-1)k,\dots,sk-1;j=i+2k,\\
0 &\hbox{otherwise.}
\end{array}\right.
\end{array}.
}
We define the first and second  matrices for $s=q$ too, but in this case the index sets are restricted to $i=qk,\dots,n;j=i-k$ and $i=(q-1)k,\dots,n-k;j=i+k$ respectively, if $p=0$ then both sets are empty and the first and second matrices are degenerate to unit matrices.
Similarly for the third and fourth matrices with $s=q-1$ the index domains are $i=qk,\dots,n;j=i-2k$ and $i=(q-2)k,\dots,n-2k;j=i+2k$ respectively and if $p=0$ then both sets are empty and the third and fourth matrices are degenerate to unit matrices.

In the above definitions the quantities $l_j^{(s-1)}, r_j^{(s-1)}, d_j^{(s-1)}$ can be obtained inductively by continuing the iterations of \eq{iterlr}, \eq{iterd} for $s=1, 2,\dots,(q-2,q-1,q)$  as follows
\vspace*{-1mm}
\Eq{iterlrs}{
\begin{array}{rl}
&l_j^{(s)}=\left\{\!\!\begin{array}{lll}
&l_j^{(s-1)}-L_{j-k}r_{j-k}^{(s-1)}/d_{j-k}^{(s-1)}\quad&(j=sk,\dots,(s+1)k-1),\\
&l_j\quad&(j=(s+1)k,\dots,n-k)
\end{array}\right.\\\\
&r_j^{(s)}=\left\{\!\!\begin{array}{lll}
&r_j^{(s-1)}-R_{j-k}l_{j-k}^{(s-1)}/d_{j-k}^{(s-1)}\quad&(j=sk,\dots,(s+1)k-1),\\
&r_j\quad&(j=(s+1)k,\dots,n-k)
\end{array}\right.
\end{array}
}
and
\vspace*{-1mm}
\Eq{iterds}{
\hspace*{14mm}d_j^{(s)}=\left\{\!\!\begin{array}{lll}
&d_j^{(s-1)}-l_{j-k}^{(s-1)}r_{j-k}^{(s-1)}/d_{j-k}^{(s-1)} \quad&(j=sk,\dots,(s+1)k-1),\\
&d_j^{(s-1)}-L_{j-2k}R_{j-2k}/d_{j-2k}^{(s-1)}&(j=(s+1)k,\dots,(s+2)k-1),\\
&d_j &(j=(s+2)k,\dots,n),\\
\end{array}\right.
\vspace*{-1mm}
}
These definitions are valid if $s< q-2.$ We put the last three  values $q-2,q-1,q$ of $s$ in parenthesis since for them the definitions should be modified. For these values of $s$ the index groups in  \eq{iterlrs}, \eq{iterds} may run out of their ranges $j=sk,\dots,n-k$ and $j=sk,\dots,n$ therefore cannot be defined, or the index groups may be restricted.

\noindent For example    \eq{iterds} is valid for  $s\le q-2$, for  $s=q-1$ the second index group  should be restricted to $j=qk,\dots, n$ if $p>0$ and empty if $p=0$ with empty third index group. For $s=q$ the diagonal elements $d_j^{(q)}$ can be defined only if $p>0$, in this case for $s=q$ in \eq{iterds}  the first index group should be restricted to $j=qk,\dots, n$ with empty second and third index groups.
\medskip

In the second process we calculate the product $A_2:=B_2^{(L)}B_2^{(l)}A_1B_2^{(r)}B_2^{(R)}$
then we continue similarly $s-2 \, (s<q-1)$ times to get
\vspace*{-1mm}
\Eq{*}{
A_{s}:=\(\prod_{j=1}^{j=s}B_{s+1-j}^{(L)}B_{s+1-j}^{(l)}\)\,A_0\,
\(\prod_{j=1}^{j=s}B_j^{(r)}B_j^{(R)}\).
\vspace*{-1mm}
}
For $A_{s}$ we have the decomposition
\Eq{*}{
A_{s}=E_{sk}\bigoplus A_{s}^*
}
where $E_{sk}\in \mathcal{M}_{sk}$ is a diagonal matrix with diagonal elements
\Eq{*}{
d_0,\dots,d_{k-1},d_{k}^{(1)},\dots,d_{2k-1}^{(1)},\dots,
d_{(s-1)k}^{(s-1)},\dots,d_{sk-1}^{(s-1)}
}
and
\vspace*{-1mm}
\Eq{*}{
A_{s}^*=A_{n+1-sk,k,2k}({\bf L}^{(s)},{\bf l}^{(s)},{\bf d}^{(s)},{\bf r}^{(s)},{\bf R}^{(s)})
}
with
\Eq{*}{
\begin{array}{lll}
&{\bf L}^{(s)}=(L_{sk},\dots,L_{n-2k})\in\mathbb{R}^{n+1-(s+2)k}\quad
&{\bf R}^{(s)}=(R_{sk},\dots,R_{n-2k})\in\mathbb{R}^{n+1-(s+2)k}\\
&\,\,{\bf l}^{(s)}=(l_{sk}^{(s-1)},\dots,l_{n-k}^{(s-1)})\in\mathbb{R}^{n+1-(s+1)k}\quad &\,\,
{\bf r}^{(s)}=(r_{sk}^{(s-1)},\dots,r_{n-k}^{(s-1)})\in\mathbb{R}^{n+1-(s+1)k}\\
&{\bf d}^{(s)}=(d_{sk}^{(s-1)},\dots,d_{n}^{(s-1)})\in\mathbb{R}^{n+1-sk}&
\end{array}
}
and the elements of these iterated diagonals are given by \eq{iterlrs}, \eq{iterds}.
\medskip

For $s=q-2$ the dimension of the vectors ${\bf L}^{(s)}, {\bf R}^{(s)}$ is $p.$ This means that if $p=0$ then $A_{q-2}^*$ is tridiagonal, and if $p>0$ then $A_{q-1}^*$ is tridiagonal. Therefore the cases $p=0$ and $p>0$ should be treated differently.

\begin{Thm} Let $1\le k< 2k\le n,$ $n+1=kq$ (i.e. $p=0$) $q>3$. Then
\Eq{*}{
A_{q-1}:=B_{q-2}^{(l)}A_{q-2}B_{q-2}^{(r)}=E_{qk}\in \mathcal{M}_{n+1}
}
is a diagonal matrix with diagonal elements
\Eq{*}{
d_0,\dots,d_{k-1},d_{k}^{(1)},\dots,d_{2k-1}^{(1)},\dots,
d_{(q-1)k}^{(q-1)},\dots,d_{qk-1}^{(q-1)}
}
\vspace*{-1mm}
defined by \eq{iterds}.
\end{Thm}

If $p>0$ then multiplying $A_{q-1}$ from the left by $B_q^{(l)}$ and the from the right by $B_q^{(r)}$   we get a diagonal matrix.
\begin{Thm} Let $1\le k< 2k\le n,$ $n+1=kq+p, \,0\le p<k,$  $q\ge 3$. Then
\Eq{*}{
A_q:=B_q^{(l)}A_{q-1}B_q^{(r)}=E_{qk+p}\in \mathcal{M}_{n+1}
}
is a diagonal matrix with diagonal elements
\Eq{*}{
d_0,\dots,d_{k-1},d_{k}^{(1)},\dots,d_{2k-1}^{(1)},\dots,
d_{(q-1)k}^{(q-1)},\dots,d_{qk-1}^{(q-1)},
d_{qk}^{(q)},\dots,d_{qk+p-1}^{(q)}
}
\vspace*{-1mm}
defined by \eq{iterds}.
\end{Thm}
For the determinant we have
\begin{Thm} Let $1\le k< 2k\le n,$ $n+1=kq+p, \,0\le p<k,$  and suppose that either $q=3,p>0$, or $q>3$. Then  the determinant of $A_{n+1,k,2k}(\bf L, l, d, r, R)$ is
\Eq{detA0}{
D_{n+1,k,2k}({\bf L, l, d, r, R})&=\prod\limits_{j=0}^{k-1}d_{j}^{(0)}d_{j+k}^{(1)}\cdots d_{j+(q-1)k}^{(q-1)}\prod\limits_{j=0}^{p-1}d_{j+qk}^{(q)}\\
&=\prod\limits_{j=0}^{p-1}d_{j}^{(0)}d_{j+k}^{(1)}\cdots d_{j+qk}^{(q)}\prod\limits_{j=p}^{k-1}d_{j}^{(0)}d_{j+k}^{(1)}
\cdots d_{j+(q-1)k}^{(q-1)}
}
where $d_{j}^{(s)}$ $(s=1,\dots,q;j=sk,\dots,n)$ are defined by \eq{iterd}, \eq{iterds}  and  $\prod\limits_{j=0}^{-1}:=1$.
\end{Thm}
To express \eq{detA0} in terms of the entries of $A_0$ is quite complicated if $q$ is large. Next we do this calculation if  $q=3$.

For $j=0,\dots,k-1$ we have
\Eq{*}{
d_{j}^{(0)}d_{j+k}^{(1)}=d_{j}\(d_{j+k}-\frac{l_jr_j}{d_j}\)=d_{j}d_{j+k}-l_jr_j
}
and
\Eq{*}{
d_{j+2k}^{(2)}&=d_{j+2k}^{(1)}-\frac{l_{j+k}^{(1)}r_{j+k}^{(1)}}{d_{j+k}^{(1)}}=
d_{j+2k}-\frac{L_jR_j}{d_j}-\frac{\(l_{j+k}-\frac{L_jr_j}{d_j}\)\(r_{j+k}-\frac{R_jl_j}{d_j}\)}
{d_{j+k}-\frac{l_jr_j}{d_j}}\\\\
&=\frac{d_jd_{j+2k}-L_jR_j}{d_j}-\frac{(l_{j+k}d_j-L_jr_j)(r_{j+k}d_j-R_jl_j)}{d_j(d_{j+k}d_j-l_jr_j)}\\\\
&=\frac{(d_jd_{j+2k}-L_jR_j)(d_{j+k}d_j-l_jr_j)-(l_{j+k}d_j-L_jr_j)(r_{j+k}d_j-R_jl_j)}{d_j(d_{j+k}d_j-l_jr_j)}\\\\
&=\frac{d_jd_{j+k}d_{j+2k}-d_{j+2k}l_jr_j-d_{j+k}L_jR_j-d_{j}l_{j+k}r_{j+k}+R_jl_jl_{j+k}
+L_jr_jr_{j+k}
}{d_jd_{j+k}-l_jr_j},
}
\Eq{*}{
d_{j}^{(0)}d_{j+k}^{(1)}d_{j+2k}^{(2)}=
d_jd_{j+k}d_{j+2k}-d_{j+2k}l_jr_j-d_{j+k}L_jR_j-d_{j}l_{j+k}r_{j+k}+R_jl_jl_{j+k}
+L_jr_jr_{j+k}.
}
If $p>0$ then for  $j=0,\dots,p$ we have to calculate $d_{j+3k}^{(3)}$ too.

\Eq{*}{
d_{j+3k}^{(3)}&=d_{j+3k}^{(2)}-\frac{l_{j+2k}^{(2)}r_{j+2k}^{(2)}}{d_{j+2k}^{(2)}}=
\frac{d_{j+3k}^{(2)}d_{j+2k}^{(2)}-l_{j+2k}^{(2)}r_{j+2k}^{(2)}}{d_{j+2k}^{(2)}}
}
where
\Eq{*}{
d_{j+3k}^{(2)}&=d_{j+3k}^{(1)}-\frac{L_{j+k}R_{j+k}}{d_{j+k}^{(1)}}=
d_{j+3k}-\frac{L_{j+k}R_{j+k}}{d_{j+k}-\frac{l_jr_j}{d_j}}=
\frac{d_jd_{j+k}d_{j+3k}-d_{j+3k}l_jr_j-d_{j}L_{j+k}R_{j+k}}{d_jd_{j+k}-l_jr_j}\\\\
l_{j+2k}^{(2)}&=l_{j+2k}^{(1)}-\frac{L_{j+k}r_{j+k}^{(1)}}{d_{j+k}^{(1)}}=
l_{j+2k}-\frac{L_{j+k}\(r_{j+k}-\frac{R_jl_j}{d_j}\)}{d_{j+k}-\frac{l_jr_j}{d_j}}\\\\
&=\frac{d_jd_{j+k}l_{j+2k}-l_{j}l_{j+2k}r_{j}-d_jL_{j+k}r_{j+k}+L_{j+k}l_jR_j}
{d_jd_{j+k}-l_jr_j}\\\\
r_{j+2k}^{(2)}&=r_{j+2k}^{(1)}-\frac{R_{j+k}l_{j+k}^{(1)}}{d_{j+k}^{(1)}}=
r_{j+2k}-\frac{R_{j+k}\(l_{j+k}-\frac{L_jr_j}{d_j}\)}{d_{j+k}-\frac{l_jr_j}{d_j}}\\\\
&=\frac{d_jd_{j+k}r_{j+2k}-r_{j}r_{j+2k}l_{j}-d_jR_{j+k}l_{j+k}+R_{j+k}r_jL_j}
{d_jd_{j+k}-l_jr_j}
}
Denote by $n_1,n_2,n_3,n_4$ the numerators by $m_1,m_2,m_3,m_4$ the  denominators of the above fraction forms of $d_{j+3k}^{(2)},d_{j+2k}^{(2)},$$l_{j+2k}^{(2)},r_{j+2k}^{(2)}$ respectively, then we have that
\Eq{*}{
d_{j+3k}^{(3)}=\frac{n_1n_2-n_3n_4}{(d_{j+k}d_j-l_jr_j)^2}\,
\frac{d_{j+k}d_j-l_jr_j}{n_2}=
\frac{n_1n_2-n_3n_4}{(d_{j+k}d_j-l_jr_j)n_2}.
}
Therefore
\Eq{*}{
d_{j}^{(0)}d_{j+k}^{(1)}d_{j+2k}^{(2)}d_{j+3k}^{(3)}=n_2d_{j+3k}^{(3)}
=\frac{n_1n_2-n_3n_4}{d_{j+k}d_j-l_jr_j}.
}
Calculating and factorizing the numerator by Maple software we obtain
\Eq{*}{
n_1n_2-n_3n_4=(d_{j+k}d_j-l_jr_j)M
}
where
\Eq{*}{
M:&=d_{j}d_{j\!+\!k}d_{j\!+\!2k}d_{j\!+\!3k}\!-\!d_{j}d_{j\!+\!k}l_{j\!+\!2k}r_{j\!+\!2k}
\!-\!d_{j}d_{j\!+\!2k}L_{j\!+\!k}R_{j\!+\!k}\!-\!d_{j}d_{j\!+\!3k}l_{j\!+\!k}r_{j\!+\!k}
\!-\!d_{j\!+\!k}d_{j\!+\!3k}L_jR_j\\\\
&-\!d_{j\!+\!2k}d_{j\!+\!3k}l_{j}r_{j}\!+\!d_{j}L_{j\!+\!k}r_{j\!+\!k}r_{j\!+\!2k}
\!+\!d_{j}R_{j\!+\!k}l_{j\!+\!k}l_{j\!+\!2k}\!+\!d_{j\!+\!3k}L_jr_{j}r_{j\!+\!k}
\!+\!d_{j\!+\!3k}R_jl_{j}l_{j\!+\!k}\\\\
&+\!L_jL_{j\!+\!k}R_jR_{j\!+\!k}\!-\!L_jR_{j\!+\!k}l_{j\!+\!2k}r_{j}
\!-\!L_{j\!+\!k}R_jl_{j}r_{j\!+\!2k}\!+\!l_{j}l_{j\!+\!2k}r_{j}r_{j\!+\!2k},
}
therefore
\Eq{*}{
d_{j}^{(0)}d_{j+k}^{(1)}d_{j+2k}^{(2)}d_{j+3k}^{(3)}=M=M_{n,k,j}({\bf L,l,d,r,R}).
}
Further let
\Eq{*}{
N_{n,k,j}&({\bf L,l,d,r,R}):=n_2\\
&=d_jd_{j+k}d_{j+2k}-d_{j+2k}l_jr_j-d_{j+k}L_jR_j-d_{j}l_{j+k}r_{j+k}+R_jl_jl_{j+k}
+L_jr_jr_{j+k}
}
then  we have
\begin{Thm} Let $1\le k< 2k\le n,$ $n+1=3k+p, \,0\le p<k$ then
\Eq{detq=3}{
D_{n+1,k,2k}({\bf L, l, d, r, R})
&=\prod\limits_{j=0}^{p-1}d_{j}^{(0)}d_{j+k}^{(1)}d_{j+2k}^{(2)}d_{j+3k}^{(3)} \prod\limits_{j=p}^{k-1}d_{j}^{(0)}d_{j+k}^{(1)}d_{j+2k}^{(2)}\\\\
&=\prod\limits_{j=0}^{p-1}M_{n,k,j}({\bf L,l,d,r,R}) \prod\limits_{j=p}^{k-1}N_{n,k,j}({\bf L,l,d,r,R})
}
where $\prod\limits_{j=0}^{-1}:=1$.
\end{Thm}

\section{$k,2k$-pentadiagonal determinants, the Toeplitz case}
In the imperfect  Toeplitz case in the quantities $M_{n,k,j},N_{n,k,j}$ of \eq{detq=3} we have to substitute $L_j=L,l_j=l,r_j=r,R_j=R$ and
\Eq{*}{
&\hbox{in $M_{n,k,j}$ for $j=0,\dots,p-1$ substitute }d_j=d-\alpha, d_{j+k}=d,d_{j+2k}=d,d_{j+3k}=d-\beta,\\
&\hbox{in $N_{n,k,j}$ for $j=p,\dots,k-1$ substitute }d_j=d-\alpha, d_{j+k}=d,d_{j+2k}=d-\beta.
}
We obtain
\begin{Thm} Let $1\le k< 2k\le n,$ $n+1=3k+p, \,0\le p<k$ then the determinant of
the imperfect pentadiagonal Toeplitz matrix $A_{n+1,k,2k}^{(\alpha,\beta)}( L, l, d, r, R)$ is
\Eq{*}{
\begin{array}{rl}
D_{n+1,k,2k}^{(\alpha,\beta)}&=\big(d^4\!-\!(\alpha+\beta)d^3-(3lr\!+\!2LR-\alpha\beta)d^2\!
+\!(2Lr^2\!+\!2Rl^2\!+\!(\alpha+\beta)(2lr\!+\!LR))d \\
&\!+\!L^2R^2\!-\!2LRlr\!+\!l^2r^2\!+\!(\alpha+\beta)(Lr^2\!+\!Rl^2)-\alpha\beta lr\big)^{p}\\
&\cdot\big(d^3\!-\!(\alpha+\beta)d^2\!-\!(2lr\!+\!LR-\alpha\beta)d\!
+\!Lr^2\!+\!Rl^2\!+\!\alpha\beta lr\big)^{k\!-\!p}
\end{array}
}
\end{Thm}
In the  Toeplitz case we immediately get by $\alpha=\beta=0$ from the previous theorem
\begin{Thm} Let $1\le k< 2k\le n,$ $n+1=3k+p, \,0\le p<k$ then the determinant of
$A_{n+1,k,2k}( L, l, d, r, R)$ is
\Eq{*}{
D_{n+1,k,2k}=&\big(d^4\!-\!(3lr\!+\!2LR)d^2\!+\!(2Lr^2\!
+\!2Rl^2)d\!+\!L^2R^2\!-\!2LRlr\!+\!l^2r^2\big)^{p}\\
&\cdot\big(d^3\!-\!(2lr\!+\!LR)d\!+\!Lr^2\!+\!Rl^2\big)^{k\!-\!p}.
}
\end{Thm}
Using a suitable rearrangement of $k,2k$-pentadiagonal (general or Toeplitz) matrices they can be reduced to the direct sum of $1,2$-pentadiagonal (general or Toeplitz) matrices. In this way we can say more about $k,2k$-pentadiagonal  determinants.

Let $n+1=kq+p$ where $0\le p<k$ and consider the  permutation $\sigma$ of the integers $0,1,\dots,n$ given by
\Eq{*}{
\begin{array}{rll}
\sigma(s+j(q+1))&=sk+j \,\, &\hbox{if }s=0,1,\dots,q;j=0,\dots,p-1,\\
\sigma(s+p+jq)&=sk+j \,\,   &\hbox{if }s=0,1,\dots,q-1;j=p,\dots,k-1.
\end{array}
}
Define the permutation matrix $P_\sigma=(p_{ij})$ by
\Eq{*}{
p_{ij}=\left\{\begin{array}{rl}
&1\hbox{ if } j=\sigma(i),\\
              &0\hbox{ otherwise. }
              \end{array}
\right.
}
Then $P_\sigma A_0P_\sigma^T$ rearranges both the rows and columns of $A_0=A_{n+1,k,2k}(\bf L,l,d,r,R)$ in the order of the permutation $\sigma$. Thus we obtain
\Eq{decomp}{
P_\sigma A_0P_\sigma^T=\(\bigoplus\limits_{j=0}^{p-1}A^{(j)}_{q+1,1,2}\)
\(\bigoplus\limits_{j=p}^{k-1}A^{(j)}_{q,1,2}\)
}
where $A^{(s)}_{t+1,1,2}({\bf L}_{(s)},{\bf l}_{(s)},{\bf d}_{(s)},{\bf r}_{(s)},{\bf R}_{(s)})$  for  $t=q,s=0,1,\dots,p-1$ and $t=q-1,s=p,p+1,\dots,k-1$ are $1,2$-pentadiagonal matrices with diagonal vectors
\Eq{*}{
\begin{array}{rll}
{\bf L}_{(s)} &=(L_s,L_{s+k},\dots,L_{s+(t-2)k}),        \quad
{\bf R}_{(s)} &=(R_s,R_{s+k},\dots,R_{s+(t-2)k}),      \\
{\bf l}_{(s)} &=(l_s,l_{s+k},\dots,l_{s+(t-1)k}),\qquad \,\,
{\bf r}_{(s)} &=(r_s,r_{s+k},\dots,r_{s+(t-1)k}),       \\
{\bf d}_{(s)} &=(d_s,d_{s+k},\dots,d_{s+tk}).        &
\end{array}
}
We remark that $\sigma$ is the same permutation as the one used by Egerv\'ary and Sz\'asz \cite{E}, see also  N. Bebiano and  S. Furtado \cite{BF2019} where a similar decomposition were given.

From \eq{decomp} it follows immediately
\begin{Thm} Let $1\le k< 2k\le n,$ then we have
\Eq{*}{
D_{n+1,k,2k}({\bf L},{\bf l},{\bf d}, {\bf r},{\bf R})=\prod\limits_{j=0}^{p-1}D^{(j)}_{q+1,1,2}\,
\prod\limits_{j=p}^{k-1}D^{(j)}_{q,1,2}
}
where $D^{(s)}_{t+1,1,2}$ denoted the determinant of the matrix $A^{(s)}_{t+1,1,2}({\bf L}_{(s)},{\bf l}_{(s)},{\bf d}_{(s)},{\bf r}_{(s)},{\bf R}_{(s)})$.
\end{Thm}

For Toeplitz matrices we get (with the notation $B^{\oplus j}:=\underbrace{B\oplus \cdots \oplus B}_{j}$)
\Eq{*}{
P_\sigma A_{n+1,k,2k}( L, l, d, r, R)P_\sigma^T=A_{q+1,1,2}( L, l, d, r, R)^{\oplus p}A_{q,1,2}( L, l, d, r, R)^{\oplus (k-p)}
}
as  in this case  $A^{(s)}_{t+1,1,2}({\bf L}_{(s)},{\bf l}_{(s)},{\bf d}_{(s)},{\bf r}_{(s)},{\bf R}_{(s)})=A_{t+1,1,2}( L, l, d, r, R)$  for  $t=q,s=0,1,\dots,p-1$ and $t=q-1,s=p,p+1,\dots,k-1$.

For the determinants we get
\Eq{*}{
D_{n+1,k,2k}( L,l,d,r,R)=D_{q+1,1,2}(L,l,d,r,R)^{p}D_{q,1,2}(L,l,d,r,R)^{k-p}
}
thus it is enough to calculate the determinants of $1,2$-pentadiagonal matrices. To do this we could use the iteration formulae  \eq{iterd}, \eq{iterds} (which are now considerably simpler) to find the diagonal elements then by \eq{detA0} to find the determinants. However it seems easier to apply existing recursion formulae for the determinants.
The six term recursion of R.A. Sweet \cite{S1969} is applicable for the determinants of Toeplitz matrices but its coefficients contain fractions and are more complicated than those of the seven term recursion found by J. Jia, B. Yang, S. Li \cite{JYL2016} thus we apply the latter.

Let  $D(n+1):=D_{n+1,1,2}(L,l,d,r,R)$ for $n\ge 2$ and let $D(-2)=0, D(-1)=0,D(0)=1,$ $D(1)=d, D(2)=d^2-lr$. For the determinant $D(3)$ we easily obtain
\Eq{initval}{
D(3)=d^3-d(LR+2lr)+Lr^2+Rl^2.
}
The recursion of \cite{JYL2016} with our notations is
\Eq{recform}{
D(n)&=dD(n-1)+(LR-lr)D(n-2)+(Lr^2+Rl^2-2dLR)D(n-3)\\
&+LR(LR-lr)D(n-4)+dL^2R^2D(n-5)-L^3R^3D(n-6)\quad(n=4,5,6,\dots).
}
Calculating $D(n)$ for $n=4,5,6,7,8,9$ using \eq{initval}, \eq{recform} with Maple software we obtain
\Eq{*}{
&d^4\!-\!d^2(2LR\!+\!3lr)\!+\!d(2Lr^2\!+\!2Rl^2)\!+\!L^2R^2\!-\!2LRlr\!+\!l^2r^2,\\\\
&d^5\!\!-\!d^3(3LR\!+\!4lr)\!+\!d^2(3Lr^2\!+\!3Rl^2)\!
+\!d(2L^2R^2\!\!-\!2LRlr\!+\!3l^2r^2)\!+\!(L^2Rr^2\!+\!R^2Ll^2\!-\!2Llr^3\!-\!2Rrl^3),\\\\
&d^6\!-\!d^4(4LR\!+\!5lr)\!+\!d^3(4Lr^2\!+\!4Rl^2)\!+\!d^2(4L^2R^2\!
+\!6l^2r^2)\!+\!d(\!-\!6Llr^3\!-\!6Rl^3r)\\
&\!+\!(\!-\!4L^2R^2lr\!+\!L^2r^4\!+\!6LRl^2r^2\!+\!R^2l^4\!-\!l^3r^3),\\\\
&d^7\!-\!d^5(5LR\!+\!6lr)\!+\!d^4(5Lr^2\!+\!5Rl^2)\!+\!d^3\!(7L^2R^2\!+\!4LRlr\!+\!10l^2r^2)\\
&\!\!-\!d^2\!(3L^2Rr^2\!\!+\!3R^2Ll^2\!\!+\!12Llr^3\!+\!12Rrl^3\!)\!
\!-\!d(2L^3R^3\!+\!6L^2R^2lr\!-\!3L^2r^4\!-\!3R^2l^4\!-\!15LRl^2r^2\!+\!4l^3r^3\!)\\
&\!+\!(3L^3R^2r^2\!+\!3R^3L^2l^2\!-\!6L^2Rlr^3\!-\!6R^2Lrl^3\!+\!3Ll^2r^4\!+\!3Rr^2l^4),\\\\
&d^8\!-\!d^6(6LR\!+\!7lr)\!+\!d^5(6Lr^2\!+\!6Rl^2)\!+\!d^4(11L^2R^2\!+\!10LRlr\!+\!15l^2r^2)\\
&\!+\!d^3(\!-\!8L^2Rr^2\!-\!8LR^2l^2\!-\!20Llr^3\!-\!20Rl^3r)\\
&\!+\!d^2(\!-\!6L^3R^3\!-\!9L^2R^2lr\!+\!6L^2r^4\!+\!24LRl^2r^2\!+\!6R^2l^4\!-\!10l^3r^3)\\
&\!+\!d(6L^3R^2r^2\!+\!6L^2R^3l^2\!-\!12L^2Rlr^3\!-\!12LR^2l^3r\!+\!12Ll^2r^4\!+\!12Rl^4r^2)\\
&\!+\!(R^4L^4\!-\!6L^3R^3lr\!+\!2L^3Rr^4\!+\!15L^2R^2l^2r^2\!-\!3L^2lr^5\!+\!2LR^3l^4
\!-\!12LRl^3r^3\!-\!3R^2l^5r\!+\!l^4r^4),\\\\
&d^9\!+\!(\!-\!7LR\!-\!8lr)d^7\!+\!(7Lr^2\!+\!7Rl^2)d^6\!+\!(16L^2R^2\!+\!18LRlr\!+\!21l^2r^2)d^5\\
&\!+\!(\!-\!15L^2Rr^2\!-\!15LR^2l^2\!-\!30Llr^3\!-\!30Rl^3r)d^4\\
&\!+\!(\!-\!13L^3R^3\!-\!16L^2R^2lr\!+\!10L^2r^4\!+\!30LRl^2r^2\!+\!10R^2l^4\!-\!20l^3r^3)d^3\\
&\!+\!(12L^3R^2r^2\!+\!12L^2R^3l^2\!-\!12L^2Rlr^3\!-\!12LR^2l^3r\!+\!30Ll^2r^4\!+\!30Rl^4r^2)d^2\\
&\!+\!(3L^4R^4\!-\!6L^3R^3lr\!+\!3L^3Rr^4\!+\!30L^2R^2l^2r^2\!-\!12L^2lr^5\!
+\!3LR^3l^4\!-\!44LRl^3r^3\!-\!12R^2l^5r\!+\!5l^4r^4)d\\
&\!+\!3L^4R^3r^2\!+\!3L^3R^4l^2\!\!-\!\!16L^3R^2lr^3\!+\!L^3r^6\!\!-\!\!16L^2R^3l^3r\!\!
+\!\!18L^2Rl^2r^4\!\!+\!\!18LR^2l^4r^2\!\!-\!\!4Ll^3r^5\!\!+\!\!R^3l^6\!\!-\!\!4Rl^5r^3.
}
Observing these determinants we see that they are monic polynomials
\Eq{*}{
p_n(L,l,d,r,R)=d^n+\sum\limits_{j=0}^{n-2} A_{n,j}(L,l,r,R)d^j
}
 of degree $n$ in $d$ where the coefficients  $A_{n,j}(L,l,r,R)\,(j=0,\dots,n-2)$  are polynomials of $L,l,r,R$ of degree $n-k$ which are symmetric in $L,R$ and $l,r$. Unfortunately for larger $n$ the formulae for these polynomials are too long. With this we have proved
\begin{Thm} Let $1\le k< 2k\le n,$ $n+1=kq+p, \,0\le p<k$ then for $q=3,4,5,6,7,8$ the determinant of
$A_{n+1,k,2k}(L,l,d,r,R)$ is
\Eq{*}{
D_{n+1,k,2k}(L,l,d,r,R)=p_{q+1}(L,l,d,r,R)^{p}\,\,p_{q}(L,l,d,r,R)^{k-p}
}
where the polynomials $p_n$ are given above and by \eq{initval}.
\end{Thm}
We conjecture that this theorem is true for all possible values of $q$.

The imperfect Toeplitz determinants $D^{(\alpha,\beta)}(n+1):=D_{n+1,k,2k}^{(\alpha,\beta)}(L,l,d,r,R)$
can be calculated by the recursion
\Eq{rec}{
D^{(\alpha,\beta)}(n)=D(n)-(\alpha+\beta)D(n-1)+\alpha\beta D(n-2)
}
found by Marr and Vineyard \cite{MV} and using the previous theorem. Another possibility is to use a nine term recursion based on \eq{rec} and \eq{recform}.

\end{document}